\documentclass[12pt]{article}
\usepackage{xcolor}
\usepackage{amsmath,amssymb,amsfonts}
\usepackage{amsthm}
\usepackage{mathrsfs}
\usepackage{graphicx}
\usepackage[verbose]{hyperref}
\hypersetup{colorlinks=false,allbordercolors=blue,pdfborderstyle={/S/U/W 1}}

\newcommand{\R}{\mathbb{R}}
\newcommand{\fr}{\displaystyle\frac}
\newcommand{\jf}{\displaystyle\int}
\newcommand{\lt}{\left}
\newcommand{\rt}{\right}
\newcommand{\be}{\begin{equation}}
\newcommand{\ee}{\end{equation}}
\newcommand{\bee}{\begin{equation*}}
\newcommand{\eee}{\end{equation*}}

\newtheorem{theorem}{Theorem}

\newtheorem{remark}[theorem]{Remark}

\begin{document}

\title{Radial Symmetry and Strict Radial Decrease for Master Equations 
with Decreasing Radial Potentials}

\author{Yahui Niu}

\date{}

\maketitle

\begin{abstract}
We study the master equation
\begin{equation*}
(\partial_t -\Delta)^{s} u(x,t) = R(|x|)f(u(x,t))\quad\mbox{in}\ \mathbb{R}^n\times\mathbb{R},
\end{equation*}
where $s\in(0,1)$, $R$ is positive and strictly decreasing in the radial variable, and $f$ is positive, locally Lipschitz, and satisfies $f'(0)<0$. Existing symmetry results for this operator require $f(0)=0$ with $f'(0)\ge 0$ (or, more generally, that $f$ be non-decreasing near the origin), because the sign of the moving-plane comparison inequality is then controlled by the nonlocal diffusion term. When $f'(0)<0$, the linear term in the comparison inequality appears at the same order as the nonlocal diffusion but with the opposite sign, destroying the comparison principle underlying the standard cut-off perturbation argument. We prove that, despite this obstruction, every positive bounded classical solution with uniform spatial decay is radially symmetric and strictly radially decreasing in $x$ for each $t$. The proof reconciles the direct method of moving planes with this adverse sign by exploiting the interplay between the spatial decay of $u$ and the strict radial monotonicity of $R$. This appears to be the first symmetry result for the master operator that accommodates a local damping mechanism at the origin.
\end{abstract}



\section{Introduction}

This paper establishes the radial symmetry and strict radial decrease of positive bounded classical solutions to the master equation
\begin{equation}\label{2.0}
(\partial_t -\Delta)^{s} u(x,t) =  R(|x|)f(u(x,t))\quad\mbox{in}\ \ \mathbb{R}^n\times\mathbb{R}.
\end{equation}
Here, $R(|x|)\in C([0,+\infty))$ is positive and strictly decreasing in $|x|$, and $f$ is positive, locally Lipschitz on $[0,\infty)$, and of class $C^1$ in a neighborhood of $0$ with
\begin{equation*}
f'(0)<-\sigma\quad\text{for some constant}\ \sigma>0.
\end{equation*}

The master operator $(\partial_t-\Delta)^s$, also known as the fully fractional heat operator, is a nonlocal pseudo-differential operator of order $s$ in time and $2s$ in space. It is defined by the singular integral
\begin{equation}\label{nonlocaloper}
(\partial_t-\Delta)^s u(x,t):=C_{n,s}\int_{-\infty}^{t}\int_{\mathbb{R}^n}\frac{u(x,t)-u(y,\tau)}{(t-\tau)^{\frac{n}{2}+1+s}}e^{-\frac{|x-y|^2}{4(t-\tau)}}\,\mathrm{d}y\,\mathrm{d}\tau,
\end{equation}
where $0<s<1$ and $C_{n,s}=\frac{1}{(4\pi)^{\frac{n}{2}}|\Gamma(-s)|}$. Two features distinguish this operator: it is nonlocal in both space and time, and it interpolates between classical operators \cite{FNW, ST}---formally reducing to $\partial_t - \Delta$ as $s\to 1$, to the fractional Laplacian $(-\Delta)^s$ when acting on functions of $x$ alone, and to the Marchaud derivative $\partial_t^s$ when acting on functions of $t$ alone.

To ensure that the singular integral in \eqref{nonlocaloper} is well-defined, we require sufficient regularity and controlled growth at infinity. We say that $u$ is a classical entire solution of \eqref{2.0} if
$$u\in \mathcal{L}(\mathbb{R}^n\times\mathbb{R})\cap C^{2s+\varepsilon,s+\varepsilon/2}_{x,t,\mathrm{loc}}(\mathbb{R}^n\times\mathbb{R})$$
for some $\varepsilon>0$, where $\mathcal{L}(\mathbb{R}^n\times\mathbb{R})$ is the space of slowly increasing functions for which the kernel in \eqref{nonlocaloper} is integrable against $|u|$ (see \cite{CG}), and $C^{2s+\varepsilon,s+\varepsilon/2}_{x,t,\mathrm{loc}}$ is the local parabolic H\"older space defined in \cite{Kry}.

Equations driven by $(\partial_t-\Delta)^s$ model anomalous diffusion with memory and long-range interactions, arising in continuous-time random walks \cite{KBS, MK}, plasma turbulence \cite{DCL1, Zaslavsky}, and magneto-thermoelasticity \cite{EE}. In many settings the medium is spatially non-uniform, and \eqref{2.0} incorporates a radial potential $R(|x|)$ that decays outward from a distinguished center. A natural question is whether solutions inherit the radial symmetry of the medium.

\textbf{A gap in the theory.} Despite growing literature on symmetry for nonlocal parabolic equations, existing results for the master operator $(\partial_t-\Delta)^s$ implicitly require that the nonlinearity satisfy $f(0)=0$ with $f'(0)\ge 0$ (or, more generally, that $f$ be non-negative and non-decreasing near the origin). This assumption is structural: it governs the sign of the differential inequality satisfied by the reflected solution and ensures that negative minima propagate in the correct direction. Consequently, nonlinearities exhibiting a local damping mechanism at low densities---that is, those with $f'(0)<0$---have remained entirely outside the scope of current symmetry theory. The present paper closes this gap.

The condition $f'(0)<-\sigma<0$ arises in physically relevant settings, including population dynamics with the strong Allee effect \cite{Courchamp2008} and combustion with Newtonian cooling \cite{Bebernes1989}. In these models the negative derivative provides a local damping mechanism that will play a central role in our moving-planes argument.

\textbf{Relation to existing methods.} For the fractional Laplacian, the direct method of moving planes was introduced by Chen, Li, and Li \cite{CLL1} and has since been widely extended (see, e.g., \cite{CL, CL2, CLL, CLZ, CM, CWNH}). For the master operator itself, Chen and Guo \cite{CG} achieved a breakthrough by introducing a space-time cut-off perturbation that forces the attainment of a minimum, thereby circumventing the difficulty that minimizing sequences may escape to infinity. More recently, Lu and Yu \cite{LY} combined this cut-off technique with the direct method of moving planes for equations with directionally increasing potentials and nonlinearities satisfying $f(0)=0$ and $f'(0)=0$. In their framework the vanishing of $f'(0)$ is a \textbf{structural prerequisite}: it guarantees that $f(u_\lambda)-f(u)$ is of higher order in the reflection difference $w_\lambda$, allowing the negative term from the master operator to dominate. When $f'(0)<0$, this higher-order cancellation is destroyed. The linear term $f'(0)w_\lambda$ now appears at the \textbf{same order} as the nonlocal diffusion but with the \textbf{opposite sign}, rendering the existing comparison argument inapplicable. Overcoming this obstruction demands not a technical adjustment but a \textbf{reconciliation of the moving-plane machinery with an adverse sign structure}---a challenge that has remained unresolved and that we address here.

\textbf{Main result.} We establish that solutions of \eqref{2.0} fully inherit the radial geometry of the equation.

\begin{theorem}\label{thm1}
Let $u\in C^{2s+\varepsilon,s+\varepsilon/2}_{x,t,\mathrm{loc}}(\mathbb{R}^n\times\mathbb{R})\cap \mathcal{L}(\mathbb{R}^n\times\mathbb{R})$ for some $\varepsilon>0$ be a positive bounded classical solution of
\begin{equation*}
(\partial_t -\Delta)^{s} u = R(|x|)f(u)
\quad \text{in}\ \mathbb{R}^n\times\mathbb{R},
\end{equation*}
where $s\in(0,1)$. Assume that $u$ is uniformly continuous and satisfies
\begin{equation}\label{cu}
\lim_{|x|\to+\infty}u(x,t)=0\quad\text{uniformly in } t\in\mathbb{R}.
\end{equation}
Suppose further that $R(r)>0$ for all $r\ge 0$, $R$ is continuous and strictly decreasing in $r$, $f>0$ is locally Lipschitz on $[0,\infty)$, $f\in C^1([0,\delta])$ for some $\delta>0$, and
\begin{equation}\label{cf}
f'(0)<-\sigma\quad\text{for some constant}\ \sigma>0.
\end{equation}
Then for each $t\in\mathbb{R}$, $u(\cdot,t)$ is radially symmetric and strictly radially decreasing with respect to the origin. Equivalently, for any unit vector $e\in\mathbb{R}^n$, $u(x,t)$ is symmetric about the plane $\{x\cdot e=0\}$ and strictly decreasing along rays emanating from the origin.
\end{theorem}

\begin{remark}
The theorem is proved by selecting an arbitrary direction as the $x_1$-axis and applying the method of moving planes. We show that the solution is symmetric with respect to $\{x_1=0\}$ and strictly monotone increasing in the $x_1$-direction for $x_1<0$. Since the direction is arbitrary, the solution enjoys full radial symmetry and strict radial decrease.
\end{remark}

\begin{remark}
The assumptions on $f$ are satisfied, for instance, by $f(u)=1-\sigma u+u^2$ with $\sigma\in(0,2)$. This function is strictly positive for all $u\ge 0$, locally Lipschitz, and satisfies $f'(0)=-\sigma<0$.
\end{remark}

\textbf{Proof strategy and novelties.} The proof of Theorem \ref{thm1} employs the direct method of moving planes combined with the cut-off perturbation technique introduced in \cite{CG} and further developed in \cite{LY}. While the overall architecture follows these works, the execution at each critical step is tailored to the structure of \eqref{2.0}. The key distinctions are:

\begin{itemize}
\item \textbf{Initial step and boundedness.} For sufficiently negative $\lambda$, the spatial decay \eqref{cu} forces the solution to be small. The condition $f'(0)<0$ then guarantees that the coefficient in the differential inequality for the reflection difference $w_\lambda$ is strictly negative, yielding an immediate sign contradiction. No additional asymptotic condition on $R$ is required. The same mechanism also drives the boundedness estimate in Step~2, preventing the minimizing sequence from escaping to infinity. Moreover, the perturbation in Step~1 uses a time-only cut-off, simplifying the estimates compared to the space-time cut-offs in \cite{CG, LY}.

\item \textbf{Limiting position.} 
The strict decrease of $R(|x|)$ becomes decisive when reaching the limiting position. The difference $R(|x^\lambda|)-R(|x|)$ is strictly positive in $\Sigma_\lambda$ and remains bounded away from zero along the minimizing sequence, providing a non-vanishing positive term that leads to a contradiction when $\lambda_0<0$. This mechanism parallels the geometric term in \cite{CG} and the coefficient difference in \cite{LY}, but is driven here by radial monotonicity.
\end{itemize}

These adaptations demonstrate how the sign of $f'(0)$, the spatial decay of $u$, and the monotonicity of $R$ each enter the proof at the appropriate stage, and provide a template for studying symmetry in fully fractional parabolic equations with $f'(0)<0$. We also note that the proof that the infimum $a_k$ tends to zero (Step~2a) is handled by a direct algebraic splitting argument relying only on the uniform continuity of $u$, thereby avoiding the case distinction used in \cite{CG}.

\textbf{Broader outlook.} While Theorem \ref{thm1} is stated for \eqref{2.0}, the strategy---combining spatial decay with the sign of $f'(0)$ for the initial contradiction and the radial monotonicity of $R$ for the limiting contradiction---transcends this setting. The same framework should apply to master equations with fractional $p$-Laplacian diffusion and cooperative systems. We leave these extensions to future work and focus here on establishing the core phenomenon: \textbf{radial symmetry persists even under local damping at the origin.}

The remainder of the paper is organized as follows. Section \ref{Pre} collects the notation for the moving planes and a key scaling estimate for the master operator. Section \ref{Proof} contains the complete proof of Theorem \ref{thm1}, carried out in three steps via the moving planes method.

\section{Preliminaries}\label{Pre}
\subsection{Notation for the moving planes}

Let $\lambda \in \R$, $$T_\lambda := \{x \in \R^n \mid x_1 = \lambda\}$$
be a moving plane perpendicular to the $x_1$-axis,
$$\Sigma_\lambda := \{x \in \R^n \mid x_1 < \lambda\}$$ be the region to the left of the hyperplane $T_\lambda$, and
$$x^\lambda := (2\lambda - x_1, x_2, \ldots, x_n)$$
be the reflection of $x$ with respect to the hyperplane $T_\lambda$.

Assume that  $u$ is a solution of the fully fractional heat equation \eqref{2.0}. To compare the values of $u(x,t)$ with $u_\lambda(x, t) := u(x^\lambda, t)$, we define $$w_\lambda(x, t) := u_\lambda(x, t) - u(x, t).$$ It is evident that $w_\lambda(x, t)$ is an antisymmetric function of $x$ with respect to the hyperplane $T_\lambda$.

\subsection{A scaling estimate}\label{estimate}
In the proof of Theorem \ref{thm1}, we will repeatedly encounter integrals of the 
form
\[
I(x_o, t_o) := \int_{-\infty}^{t_o} \int_{\Sigma_{\lambda}} 
\frac{1}{(t_o - \tau)^{\frac{n}{2} + 1 + s}} 
e^{-\frac{|x_o - y^{\lambda}|^{2}}{4(t_o - \tau)}} 
\mathrm{d}y \mathrm{d}\tau,
\]
where \((x_o, t_o) \in \Sigma_{\lambda} \times \R\). 
Since the reflection
\(y \mapsto y^{\lambda}\) preserves the Lebesgue measure and maps
\(\Sigma_{\lambda}\) onto \(\Sigma_{\lambda}^{\mathrm{C}}\),
we obtain
\[
I(x_o, t_o) = \int_{-\infty}^{t_o} \int_{\Sigma_{\lambda}^C}
\frac{1}{(t_o - \tau)^{\frac{n}{2} + 1 + s}}
e^{-\frac{|x_o - y|^{2}}{4(t_o - \tau)}}
\,\mathrm{d}y\,\mathrm{d}\tau .
\]
To evaluate this integral, we make the change of variables 
$$
y=x_o+r z,\ \ \tau=t_o-r^2\rho
$$
with $r = \operatorname{dist}(x_o, T_{\lambda})$, then 
$\operatorname{d}\!y=r^n\operatorname{d}\!z$, $\operatorname{d}\!\tau=-r^2 \operatorname{d}\!\rho,$
and the domain $\Sigma_{\lambda}^C$ corresponds to $\{z\in\R^n\mid z_1>1\}$.
Consequently,
$$
\begin{aligned}
    I(x_o, t_o)
    &=\jf_{-\infty}^{t_o}\jf_{\Sigma_\lambda^C}\frac{1}{(t_o-\tau)^{\frac{n}{2}+1+s}}e^{-\frac{|x_o-y|^2}{4(t_o-\tau)}}\operatorname{d}\!y\operatorname{d}\!\tau   \\
    & =\frac{1}{r^{2s}}\jf_{0}^{+\infty}\jf_{\{z_1>1\}}\frac{1}{\rho^{\frac{n}{2}+1+s}}e^{-\frac{|z|^2}{4\rho}}
    \operatorname{d}\!z\operatorname{d}\!\rho. 
\end{aligned}
$$
 The double integral on the right-hand side is an absolute constant depending only on $n$ and $s$. Throughout the proof, we absorb this constant into generic constants denoted by $C,\ C_1,\ C_2,$ etc.

\section{Proof of Theorem \ref{thm1}}\label{Proof}
In this section, we present the proof of Theorem \ref{thm1}. 
By fixing an arbitrary direction as the $x_1$-axis, we apply 
the direct method of moving planes to show that every positive 
solution is strictly increasing in the $x_1$-direction for 
$x_1<0$. The proof is accomplished in three steps.

In Step 1, we show that for $\lambda$ sufficiently negative, it holds
\be \label{mp} w_\lambda(x,t) \geq 0, \;\forall\ (x,t)\in\Sigma_\lambda\times\R.\ee
This provides a starting point to move the plane.
\smallskip

In Step 2, we move the plane $T_\lambda$ along the $x_1$ direction as long as the above inequality holds.
We prove that this plane can be moved all the way to $\lambda = 0,$ that is, \eqref{mp} holds for all
real numbers $\lambda\leq 0$.
\smallskip

In Step 3, we further derive the strict inequality
$$  w_\lambda(x,t) > 0, \;\forall\ (x,t)\in\Sigma_\lambda\times\R,\;\forall\ \lambda\leq 0,$$
which implies that for each fixed $t$, $u(x,t)$ is strictly increasing in $x_1$ direction. 
\medskip

Now we carry out the details. 
\begin{proof}[Proof of Theorem \ref{thm1}] \,

\bigskip

It follows from equation \eqref {2.0},   the strict monotonicity of $R(|x|)$ and the positivity of 
 $f$ that
\be\label{maineq}
  \left.
\begin{array}{ll}
    (\partial_t-\Delta)^s w_\lambda(x,t)&= R(|x^\lambda|)f(u_\lambda)- R(|x|)f(u)\\[0.2cm]
    &=[R(|x^\lambda|)-R(|x|)]f(u_\lambda)+R(|x|)[f(u_\lambda)-f(u)]\\[0.2cm]
    &\geq R(|x|)[f(u_\lambda)-f(u)].\\[0.2cm]
\end{array}
\right.
\ee
 We want to show that
\begin{equation*}
w_\lambda (x,t) \geq 0, \;\forall (x,t)\in\Sigma_\lambda\times\R,\;\forall\lambda\leq0.
\end{equation*}

\noindent\textbf{Step 1.} We show that for $\lambda$ sufficiently negative, \eqref{mp} holds.

We argue by contradiction. Suppose that \eqref{mp} fails for some $\lambda$ that is sufficiently negative.
By the uniform spatial decay~\eqref{cu}, we can choose $\lambda$ so negative that $|x|>-\lambda$ is large
for all $x\in\Sigma_\lambda$, and consequently
\begin{equation}\label{smallu}
u(x,t)<\delta\qquad \forall\,(x,t)\in\Sigma_\lambda\times\R.
\end{equation}
(Indeed, by~\eqref{cu}, there exists $R_\delta>0$ such that $|x|>R_\delta$
implies $u(x,t)<\delta$. Take $\lambda < -R_\delta$; then for any $x\in\Sigma_\lambda$ we have $x_1<\lambda<-R_\delta$,
hence $|x|\ge|x_1|>R_\delta$, yielding $u(x,t)<\delta$.)

Now, because $u$ is bounded and $w_\lambda$ takes negative values somewhere in $\Sigma_\lambda\times\R$,
there exists a constant $m>0$ such that
\begin{equation}\label{NRP5}
\inf_{(x,t)\in \Sigma_\lambda\times \R}w_\lambda(x,t)=:-m<0.
\end{equation}
Since the domain is unbounded, a minimum may not be attained; we therefore use a perturbation argument.

By the definition of the infimum, there exists a sequence 
$\{(x^k,t_k)\}\subset\Sigma_\lambda\times\R$ such that
\begin{equation*}
  w_\lambda(x^k,t_k)=: -m_k \to -m \quad\text{as } k\to\infty.
\end{equation*}
Let $\varepsilon_k:=m-m_k$, then $\varepsilon_k>0$ and 
$\varepsilon_k\to 0$ as $k\to\infty$.

We claim that $\{x^k\}$ must be bounded. 
Indeed, if $|x^k|\to\infty$ along a subsequence, then 
by the uniform spatial decay \eqref{cu} we would have 
$u(x^k,t_k)\to0$ and $u_\lambda(x^k,t_k)\to0$, which forces 
$w_\lambda(x^k,t_k)\to0$, contradicting 
$w_\lambda(x^k,t_k)\to -m<0$. 
Therefore the negative infimum cannot be approached by sequences escaping to spatial infinity. The only obstacle that may prevent the infimum from being attained is the unboundedness of the time variable $t\in\R$. 
To overcome this, we introduce an auxiliary function that 
perturbs only in the time variable:
\begin{equation}\label{2.5}
  v_k(x,t):=w_\lambda(x,t)-\varepsilon_k\eta_k(t),
\end{equation}
where
$$\eta_k(t)=\eta\Big(\frac{t-t_k}{r_k^2}\Big)$$
with $r_k=\frac{1}{2}{\rm {dist}}(x^k,T_\lambda)$ and $\eta\in C_0^\infty(\R)$ is a smooth cut-off function satisfying
\begin{equation*}
\left\{\begin{array}{r@{\ \ }c@{\ \ }ll}
0\leq \eta\leq 1 &\mbox{in}&\ \ \R\,, \\[0.05cm]
\eta= 1 &\mbox{in}&\ \ [-1,1]\,, \\[0.05cm]
\eta= 0 &\mbox{in}&\ \ \R\backslash[-2,2]\,. \\[0.05cm]
\end{array}\right.
\end{equation*}

First, by the definition of $m_k$ and $\varepsilon_k$, we have
\begin{equation}\label{vk_at_xk}
  v_k(x^k,t_k)=w_\lambda(x^k,t_k)-\varepsilon_k=-m_k-m+m_k=-m.
\end{equation}
Second, if $|t-t_k|\geq2r^2_k$ and $x\in\Sigma_\lambda$, then it follows from \eqref{NRP5} that
\begin{equation}\label{vk_outside_strip}
  v_k(x,t)=w_\lambda(x,t)\geq-m.
\end{equation}
Third, the uniform spatial decay $\lim_{|x|\to\infty}u(x,t)=0$ implies 
$\lim_{|x|\to\infty}w_\lambda(x,t)=0$ uniformly in $t$, and consequently
\begin{equation}\label{vk_spatial_decay}
\lim_{|x|\to\infty}v_k(x,t) = -\varepsilon_k\eta_k(t) \ge -\varepsilon_k.
\end{equation}
Now we combine these three facts.
Since $\varepsilon_k \to 0$, we have $-\varepsilon_k > -m$ for all $k$ large.
Thus \eqref{vk_spatial_decay} implies $v_k(x,t) > -m$ for $|x|$ sufficiently large,
so the infimum of $v_k$ cannot be approached by sequences escaping to spatial infinity.
Outside the time strip $|t-t_k| \ge 2r_k^2$, \eqref{vk_outside_strip} gives $v_k \ge -m$,
while \eqref{vk_at_xk} already realizes the value $-m$ at $(x^k,t_k)$,
which lies inside that strip.
Hence the global minimum of $v_k$ over $\Sigma_\lambda \times \R$ must be attained in a
compact subset of $\Sigma_\lambda \times [t_k-2r_k^2,\, t_k+2r_k^2]$. Consequently, there exists $(\bar{x}^k,\bar{t}_k)$ in that set such that
\begin{equation}\label{NRP6}
 -m-\varepsilon_k\leq v_k(\bar{x}^k,\bar{t}_k)= \inf_{\Sigma_\lambda\times\R}v_k(x,t)\leq-m,
\end{equation}
and 
\begin{equation}\label{NRP7}
 -m\leq w_\lambda(\bar{x}^k,\bar{t}_k)\leq-m_k.
\end{equation}

We already know that $\{x^k\}$ is bounded.
If $|\bar{x}^k| \to \infty$ along a subsequence, then 
$w_\lambda(\bar{x}^k,\bar{t}_k) \to 0$ by the spatial decay \eqref{cu}, 
contradicting $w_\lambda(\bar{x}^k,\bar{t}_k) \le -m_k \to -m <0$.
Thus $\{\bar{x}^k\}$ is also bounded.
Consequently, there exists $M>0$ such that 
\begin{equation}\label{xk}
|\bar{x}^k - x^k| \le M\quad\text{for all large} k.
\end{equation} 

Recalling $r_k = \frac12 \operatorname{dist}(x^k, T_\lambda)$ and define
$\bar{r}_k := \operatorname{dist}(\bar{x}^k, T_\lambda)$, 
the $x_1$-coordinates satisfy 
$x_1^k = \lambda - 2r_k$ and $\bar{x}_1^k = \lambda - \bar{r}_k$. Hence
\[
\bar{r}_k - 2r_k = -(\bar{x}_1^k - x_1^k),
\]
combining with \eqref{xk}, we obtain
\[
2r_k - M \le \bar{r}_k \le 2r_k + M.
\]
In particular, $\bar{r}_k$ is also bounded from above and below by 
positive constants, and the two quantities $r_k^{2s}$ and $\bar{r}_k^{2s}$ 
are interchangeable up to a multiplicative constant.

In addition, starting from the definition of operator $(\partial_t-\Delta)^s$ as well as  the fact that 
$$|\bar{x}^k-y^\lambda|>|\bar{x}^k-y|\ \ \mbox{ for }y\in \Sigma_\lambda,$$  we arrive at
\begin{equation}\label{explain}
\left.\begin{array}{r@{\ \ }c@{\ \ }ll}
(\partial_t-\Delta)^s v_k(\bar{x}^k,\bar{t}_k)&=&C_{n,s}\jf_{-\infty}^{\bar{t}_k}\jf_{\R^{n}}\frac{v_k(\bar{x}^k,\bar{t}_k)-v_k(y,\tau)}{(\bar{t}_k-\tau)^{\frac{n}{2}+1+s}}e^{-\frac{|\bar{x}^k-y|^2}{4(\bar{t}_k-\tau)}}\operatorname{d}\!y\operatorname{d}\!\tau \\[0.3cm]
&=&C_{n,s}\jf_{-\infty}^{\bar{t}_k}\jf_{\Sigma_\lambda}\frac{v_k(\bar{x}^k,\bar{t}_k)-v_k(y,\tau)}{(\bar{t}_k-\tau)^{\frac{n}{2}+1+s}}e^{-\frac{|\bar{x}^k-y|^2}{4(\bar{t}_k-\tau)}}\operatorname{d}\!y\operatorname{d}\!\tau \\[0.3cm]&\,&\,+C_{n,s}\jf_{-\infty}^{\bar{t}_k}\jf_{\Sigma_\lambda}\frac{v_k(\bar{x}^k,\bar{t}_k)-v_k(y^\lambda,\tau)}{(\bar{t}_k-\tau)^{\frac{n}{2}+1+s}}e^{-\frac{|\bar{x}^k-y^\lambda|^2}{4(\bar{t}_k-\tau)}}\operatorname{d}\!y\operatorname{d}\!\tau \\[0.3cm]
&\leq&C_{n,s}\jf_{-\infty}^{\bar{t}_k}\jf_{\Sigma_\lambda}\frac{2v_k(\bar{x}^k,\bar{t}_k)-v_k(y,\tau)-v_k(y^\lambda,\tau)}{(\bar{t}_k-\tau)^{\frac{n}{2}+1+s}}e^{-\frac{|\bar{x}^k-y^\lambda|^2}{4(\bar{t}_k-\tau)}}\operatorname{d}\!y\operatorname{d}\!\tau \\[0.3cm]
&\leq&C_{n,s}2\lt(v_k(\bar{x}^k,\bar{t}_k)+\varepsilon_k\rt)\jf_{-\infty}^{\bar{t}_k}\jf_{\Sigma_\lambda}\frac{1}{(\bar{t}_k-\tau)^{\frac{n}{2}+1+s}}e^{-\frac{|\bar{x}^k-y^\lambda|^2}{4(\bar{t}_k-\tau)}}\operatorname{d}\!y\operatorname{d}\!\tau,\end{array}\right.
\end{equation}
where we have used $$
v_k(y,\tau)+v_k(y^\lambda,\tau)=-2\varepsilon_k\eta_k(\tau)\geq-2\varepsilon_k,
$$since $w_{\lambda}$ is antisymmetric, $w_{\lambda}(y,\tau) + w_{\lambda}(y^{\lambda},\tau) = 0.$ 

Applying the estimate from Section \ref{estimate} with $x_{o} = \bar{x}^{k}$, 
$t_{o} = \bar{t}_{k}$ and $r = \bar{r}_{k}={\rm {dist}}(\bar{x}^k,T_\lambda)$, we obtain from \eqref{NRP6} and \eqref{explain} that
\begin{equation}\label{2.8}
(\partial_t-\Delta)^s v_k(\bar{x}^k,\bar{t}_k)\le \frac{C(-m + \varepsilon_{k})}{\bar{r}_{k}^{2s}}.
\end{equation}
Hence, a combination of \eqref{2.5} and  \eqref{2.8} yields  that
\begin{align}\label{mp1}
(\partial_t-\Delta)^sw_\lambda(\bar{x}^k,\bar{t}_k) 
&=(\partial_t-\Delta)^s v_k(\bar{x}^k,\bar{t}_k)+\varepsilon_k(\partial_t-\Delta)^s \eta_k(\bar{t}_k)\notag\\  &\leq\fr{C_1(-m+\varepsilon_k)}{\bar{r}_k^{2s}}+\fr{C_2\varepsilon_k}{r_k^{2s}},
\end{align}
where we have used the following  estimate
\[(\partial_t-\Delta)^s \eta_k(\bar{t}_k)\leq \fr{C_2}{r_k^{2s}},\]
which follows by the same scaling argument as in~\cite[Corollary~2.2]{CM1}.

As noted above, $r_{k}^{2s}$ and $\bar{r}_{k}^{2s}$ are 
comparable. Moreover, $\varepsilon_{k} \to 0$ and thus for large $k$, $\varepsilon_{k}<m/2$ and $-m + \varepsilon_{k} \le -m/2 < 0$. Therefore, since $\varepsilon_k$ can be absorbed into the negative term 
and the constants from comparability can be incorporated into a new constant, we conclude 
from \eqref{mp1} and $-m<-m_{k}$ that for all sufficiently large $k$,
\begin{equation}\label{eq:3.11}
(\partial_{t} - \Delta)^{s} w_{\lambda}(\bar{x}^{k},\bar{t}_{k})\le \frac{C(-m)}{\bar{r}_{k}^{2s}}
\le \frac{C(-m_{k})}{\bar{r}_{k}^{2s}},
\end{equation}
where $C>0$ is a generic constant.

On the other hand, from~\eqref{smallu} we already have $u(\bar{x}^k,\bar{t}_k)<\delta$.
Because $w_\lambda(\bar{x}^k,\bar{t}_k)<0$, it follows that
$u_\lambda(\bar{x}^k,\bar{t}_k) < u(\bar{x}^k,\bar{t}_k) < \delta$.
Since $f$ is differentiable on $[0,\delta]$, we may apply the mean value theorem
to the term $f(u_\lambda)-f(u)$ and, together with~\eqref{maineq} and~\eqref{eq:3.11}, obtain
\begin{equation}
\label{mp2}
\begin{split}
\fr{C(-m_k)}{\bar{r}_k^{2s}}
 &\geq (\partial_{t} - \Delta)^{s} w_{\lambda}(\bar{x}^{k},\bar{t}_{k})\\
 &\geq  R(|\bar{x}^k|)\Big[f\big(u_\lambda(\bar{x}^k,\bar{t}_k)\big)-f\big(u(\bar{x}^k,\bar{t}_k)\big)\Big]\\
 &=
R(|\bar{x}^k|)f'\big(\xi_\lambda(\bar{x}^k,\bar{t}_k)\big)w_\lambda(\bar{x}^k,\bar{t}_k)
\end{split}
\end{equation}
where $\xi_\lambda(\bar{x}^k,\bar{t}_k)$ satisfies $u_\lambda(\bar{x}^k,\bar{t}_k)<\xi_\lambda(\bar{x}^k,\bar{t}_k)<u(\bar{x}^k,\bar{t}_k)<\delta$. So we can deduce from \eqref{cf} and $f\in C^1([0,\delta])$ that there exists a constant $\alpha>0$
such that 
\begin{equation}
\label{fs}
f'\big(\xi_\lambda(\bar{x}^k,\bar{t}_k)\big)<-\alpha \ \ \text{for $k$ large enough.}
\end{equation}
As a result, combining \eqref{NRP7}, \eqref{fs}, we see that the left-hand side of \eqref{mp2} is strictly negative because $-m_k<0$ and $\bar{r}_k^{2s}>0$, while the right-hand side satisfies
\[
R(|\bar{x}^k|)\,f'\bigl(\xi_\lambda(\bar{x}^k,\bar{t}_k)\bigr)\,w_\lambda(\bar{x}^k,\bar{t}_k) >0,
\]
since $R(|\bar{x}^k|)>0$, $f'\bigl(\xi_\lambda(\bar{x}^k,\bar{t}_k)\bigr)<-\alpha<0$ and $w_\lambda(\bar{x}^k,\bar{t}_k)<0$.
This contradiction shows that our assumption was false; hence
 $$ w_\lambda(x,t) \geq 0, \;\forall\ (x,t)\in\Sigma_\lambda\times\R$$
 for all sufficiently negative $\lambda$. This completes Step~1.
\medskip

\bigskip

\noindent\textbf{Step 2.} Step 1 provides a starting point to move the plane.
Now we move the plane $T_\lambda$ towards the right along the $x_1$-direction
as long as the inequality \eqref{mp} holds, up to its limiting position
$T_{\lambda_0}$ with $\lambda_0$ defined by
\[
\lambda_0:=\sup\{\,\lambda\leq 0\mid w_\mu\geq 0\ \text{in}\ \Sigma_\mu\times\R,\
\forall\,\mu\leq\lambda\,\}.
\]
We will show that
\[
\lambda_0=0.
\]
Suppose $\lambda_0<0$. By the definition of $\lambda_0$, there exist a sequence
$\{\lambda_k\}\searrow\lambda_0$ and positive numbers $a_k$ such that
\[
\inf_{(x,t)\in\Sigma_{\lambda_k}\times\R} w_{\lambda_k}(x,t)=:-a_k<0.
\]

\medskip

\noindent\textit{Step 2a. The infimum tends to zero.}
We claim that
\begin{equation}\label{mf4}
a_k\to 0 \quad\text{as }k\to\infty.
\end{equation}
Suppose the contrary. Then there exist $\varepsilon_0>0$ and a subsequence
(still denoted by $\{\lambda_k\}$) such that $a_k\ge\varepsilon_0$ for all $k$.
For each such $k$, pick $(z^k,\tau_k)\in\Sigma_{\lambda_k}\times\R$ satisfying
\begin{equation}\label{wkneg}
w_{\lambda_k}(z^k,\tau_k)\le -a_k+\frac{\varepsilon_0}{2}\le -\frac{\varepsilon_0}{2}<0.
\end{equation}
If $|z^k|\to\infty$ along a further subsequence, then the uniform spatial
decay~\eqref{cu} forces $w_{\lambda_k}(z^k,\tau_k)\to0$, contradicting~\eqref{wkneg}.
Hence $\{z^k\}$ is bounded. Passing to a subsequence, we may assume
$z^k\to z^0\in\overline{\Sigma}_{\lambda_0}$.

We claim that $z_1^0<\lambda_0$. Indeed, if $z_1^0=\lambda_0$, then
$|z^k-(z^k)^{\lambda_k}|\to0$ as $k\to\infty$. By the uniform continuity of $u$,
this implies $w_{\lambda_k}(z^k,\tau_k)\to0$, again contradicting~\eqref{wkneg}.
Thus $z_1^0<\lambda_0$, and for all sufficiently large $k$ we have
$z^k\in\Sigma_{\lambda_0}$.

Now, using the uniform continuity of $u$ once more,
\[
\begin{aligned}
w_{\lambda_0}(z^k,\tau_k)
&= \bigl[u(z^{k^{\lambda_0}},\tau_k)-u(z^{k^{\lambda_k}},\tau_k)\bigr] + w_{\lambda_k}(z^k,\tau_k)\\
&\le |u(z^{k^{\lambda_0}},\tau_k)-u(z^{k^{\lambda_k}},\tau_k)| - \frac{\varepsilon_0}{2}.
\end{aligned}
\]
Since $|z^{k^{\lambda_0}}-z^{k^{\lambda_k}}| = 2|\lambda_k-\lambda_0|\to0$, the
first term on the right tends to $0$ uniformly in $k$. Hence for large $k$,
\[
w_{\lambda_0}(z^k,\tau_k)\le -\frac{\varepsilon_0}{4}<0,
\]
which contradicts $w_{\lambda_0}\ge0$ in $\Sigma_{\lambda_0}\times\R$.
Therefore, \eqref{mf4} holds.

\medskip
\noindent\textit{Step 2b. Construction of the perturbed minimizer.}
By the definition of the infimum and~\eqref{mf4}, for each sufficiently large $k$
there exists $(y^k,s_k)\in\Sigma_{\lambda_k}\times\R$ such that
\[
w_{\lambda_k}(y^k,s_k) \le -a_k + a_k^2 < 0.
\]
To ensure that a minimum is attained, we perturb $w_{\lambda_k}$ near $(y^k,s_k)$
by setting
\[
v_k(x,t):=w_{\lambda_k}(x,t)-a_k^2\psi_k(x,t),\qquad (x,t)\in\R^n\times\R,
\]
where
\[
\psi_k(x,t)=\psi\Big(\frac{x-y^k}{d_k},\frac{t-s_k}{d_k^2}\Big),
\]
with $d_k=\frac12\operatorname{dist}(y^k,T_{\lambda_k})>0$ and
$\psi\in C^\infty_0(\R^n\times\R)$ a smooth cut-off function satisfying
\[
\begin{cases}
0\leq \psi\leq 1 &\text{in }\R^n\times\R,\\[2pt]
\psi= 1 &\text{in }B_{1/2}(0)\times\big[-\frac12,\frac12\big],\\[2pt]
\psi= 0 &\text{in }\big(\R^n\times\R\big)\setminus
\big(B_{1}(0)\times[-1,1]\big).
\end{cases}
\]
Denote
\[
P_k(y^k,s_k):=B_{d_k}(y^k)\times\big[s_k-d_k^{2},\,s_k+d_k^{2}\big]
\subset\Sigma_{\lambda_k}\times\R,
\]
the parabolic cylinder centered at $(y^k,s_k)$.
Since
\[
v_k(y^k,s_k)\le -a_k,\qquad
v_k(x,t)=w_{\lambda_k}(x,t)\ge -a_k\quad\text{in }
\big(\R^n\times\R\big)\setminus P_k(y^k,s_k),
\]
each $v_k$ must attain its minimum, which is at most $-a_k$, in
$\overline{P_k(y^k,s_k)}\subset\overline{\Sigma}_{\lambda_k}\times\R$.
That is, there exists $(\bar{y}^k,\bar{s}_k)\in\overline{P_k(y^k,s_k)}$ such that
\[
-a_k-a_k^2 \le v_k(\bar{y}^k,\bar{s}_k)=\inf_{\Sigma_{\lambda_k}\times\R} v_k(x,t)
\le -a_k.
\]
Consequently,
\begin{equation}\label{mf5}
-a_k \le w_{\lambda_k}(\bar{y}^k,\bar{s}_k) \le -a_k+a_k^2 < 0.
\end{equation}

Since $(\bar{y}^k,\bar{s}_k)\in\overline{P_k(y^k,s_k)}$, we have
$|\bar{y}^k-y^k|\le d_k$. Let $\bar{d}_k:=\operatorname{dist}(\bar{y}^k,T_{\lambda_k})$.
Then
\[
d_k \le \bar{d}_k \le 3d_k,
\]
so $d_k$ and $\bar{d}_k$ are comparable and may be interchanged in the estimates
below up to multiplicative constants.

By a standard scaling argument (cf.\ \cite[Corollary~2.2]{CM1}),
\[
(\partial_t-\Delta)^s\psi_k(\bar{y}^k,\bar{s}_k)\le \frac{C}{d_k^{2s}},
\]
with $C$ depending only on $n$ and $s$. Applying the same minimum argument
as in Step~1 then yields
\begin{equation}\label{eq:wkest}
(\partial_t-\Delta)^s w_{\lambda_k}(\bar{y}^k,\bar{s}_k)
\le \frac{C(-a_k+a_k^2)}{\bar{d}_k^{2s}}.
\end{equation}

\medskip
\noindent\textit{Step 2c. Boundedness of the minimizing sequence.}
Combining~\eqref{eq:wkest} with the differential inequality~\eqref{maineq} we obtain
\begin{equation}\label{ak}
\frac{C(-a_k+a_k^2)}{\bar{d}_k^{2s}}
\ge R(|\bar{y}^k|)\Big[f\big(u_{\lambda_k}(\bar{y}^k,\bar{s}_k)\big)
-f\big(u(\bar{y}^k,\bar{s}_k)\big)\Big].
\end{equation}

We prove that $\{|\bar{y}^k|\}$ is bounded. Suppose not; then along a subsequence
$|\bar{y}^k|\to+\infty$. Since $\lambda_k\to\lambda_0$ is bounded, we also have
$|(\bar{y}^k)^{\lambda_k}|\to+\infty$. By the uniform spatial decay~\eqref{cu},
for all sufficiently large $k$ we have
\[
u(\bar{y}^k,\bar{s}_k) < \delta,\qquad
u_{\lambda_k}(\bar{y}^k,\bar{s}_k) < \delta.
\]
Since $f\in C^1([0,\delta])$, the mean value theorem gives
\[
f(u_{\lambda_k}(\bar{y}^k,\bar{s}_k)) - f(u(\bar{y}^k,\bar{s}_k)) = f'(\xi_k(\bar{y}^k,\bar{s}_k))\, w_{\lambda_k}(\bar{y}^k,\bar{s}_k)
\]
with $\xi_k$ between $u_{\lambda_k}$ and $u$, and consequently by \eqref{cf}
$f'(\xi_k) \le -\sigma/2 < 0$ for $k$ large enough. Then recall $R>0$ and $w_{\lambda_k}<0$ by~\eqref{mf5}, the right-hand side of~\eqref{ak} is strictly positive. But the left-hand side
$\frac{C(-a_k+a_k^2)}{\bar{d}_k^{2s}}$ is strictly negative (since $a_k\to0$ and
$-a_k<0$ for large $k$). This contradiction shows $\{|\bar{y}^k|\}$ is bounded.

\medskip
\noindent\textit{Step 2d. Positive lower bound for $u$.}
We claim that $u(\bar{y}^k,\bar{s}_k)\ge c>0$ for all large $k$.
Suppose instead that along a subsequence $u(\bar{y}^k,\bar{s}_k)\to0$.
From~\eqref{mf4} and~\eqref{mf5} we have $w_{\lambda_k}(\bar{y}^k,\bar{s}_k)\to0$,
hence $u_{\lambda_k}(\bar{y}^k,\bar{s}_k)\to0$. Applying the mean value theorem again,
\[
f(u_{\lambda_k})-f(u) = f'(\xi_k)\,w_{\lambda_k}
\]
with $f'(\xi_k)\to f'(0)<0$. The right-hand side of~\eqref{ak} is therefore
strictly positive, while the left-hand side is strictly negative --- a contradiction.
Hence
\begin{equation}\label{positiveu}
u(\bar{y}^k,\bar{s}_k)\ge c > 0 \qquad\text{for all sufficiently large }k.
\end{equation}

\medskip
\noindent\textit{Step 2e. Contradiction via the full differential inequality.}
Now we return to the complete differential inequality~\eqref{maineq}. Instead of
discarding the term $[R(|x^{\lambda_k}|)-R(|x|)]f(u_{\lambda_k})$ as was done
in~\eqref{ak}, we keep it and obtain
\begin{equation}\label{eq:3.19}
\begin{aligned}
\frac{C(-a_k+a_k^2)}{\bar{d}_k^{2s}}
&\ge \Big[R\big(|(\bar{y}^k)^{\lambda_k}|\big)-R(|\bar{y}^k|)\Big]
      f\big(u_{\lambda_k}(\bar{y}^k,\bar{s}_k)\big) \\
&\quad + R(|\bar{y}^k|)\Big[f\big(u_{\lambda_k}(\bar{y}^k,\bar{s}_k)\big)
                      -f\big(u(\bar{y}^k,\bar{s}_k)\big)\Big].
\end{aligned}
\end{equation}

Since $\{|\bar{y}^k|\}$ is bounded and $\lambda_k\to\lambda_0<0$, we may pass
to a subsequence (still denoted by $\bar{y}^k$) such that
$\bar{y}^k\to y^*\in\overline{\Sigma}_{\lambda_0}$.
From $\operatorname{dist}(\bar{y}^k,T_{\lambda_k})\ge d_k$ and the fact that
$d_k$ is bounded away from zero (otherwise the left-hand side of~\eqref{ak}
would blow up while the right-hand side remains bounded), we infer
$\lambda_k-\bar{y}_1^k\ge d_k\ge c>0$, which yields $y_1^*\le\lambda_0-c<\lambda_0$ in the limit.
Consequently, both $y^*$ and its reflection $(y^*)^{\lambda_0}$ lie in the
half-space $\{x_1<\lambda_0\}$ and satisfy
\[
|(y^*)^{\lambda_0}| = \sqrt{(2\lambda_0-y_1^*)^2+|y'|^2}
< \sqrt{(y_1^*)^2+|y'|^2} = |y^*|
\]
because $2\lambda_0-y_1^*<y_1^*<0$ (here $y'=(y_2,\ldots,y_n)$).
By the continuity of the reflection and the convergence of $\bar{y}^k$,
this strict inequality is inherited for large $k$: there exists $\delta>0$ such that
\[
|(\bar{y}^k)^{\lambda_k}| \le |\bar{y}^k| - \delta
\qquad\text{for all sufficiently large }k.
\]
Since $R$ is strictly decreasing, we obtain
\begin{equation}\label{eq:3.20}
R\big(|(\bar{y}^k)^{\lambda_k}|\big) - R(|\bar{y}^k|) \ge C > 0
\qquad\text{for all sufficiently large }k.
\end{equation}

From~\eqref{positiveu} and~\eqref{mf5} we have
\[
u_{\lambda_k}(\bar{y}^k,\bar{s}_k)
= u(\bar{y}^k,\bar{s}_k) + w_{\lambda_k}(\bar{y}^k,\bar{s}_k)
\ge c - |w_{\lambda_k}| \ge c/2 > 0
\]
for large $k$. Since $f>0$ is continuous,
\begin{equation}\label{eq:3.21}
f\big(u_{\lambda_k}(\bar{y}^k,\bar{s}_k)\big) \ge C > 0
\end{equation}
uniformly in large $k$.

Now let $k\to+\infty$ in~\eqref{eq:3.19}. The left-hand side tends to $0$
(since $a_k\to0$ and $\bar{d}_k$ is bounded). The second term on the right-hand side
tends to $0$ because
\[
\big|f(u_{\lambda_k})-f(u)\big|
\le L|u_{\lambda_k}-u|
= L|w_{\lambda_k}|
\le L a_k \to 0
\]
by the local Lipschitz continuity of $f$ and~\eqref{mf5}.
However, by~\eqref{eq:3.20} and~\eqref{eq:3.21}, the first term satisfies
\[
\big[R\big(|(\bar{y}^k)^{\lambda_k}|\big) - R(|\bar{y}^k|)\big]\,
f\big(u_{\lambda_k}(\bar{y}^k,\bar{s}_k)\big) \ge C > 0
\]
for all large $k$. This yields a contradiction.

Therefore $\lambda_0=0$, and the plane can be moved all the way to the origin.
This completes Step~2.

\bigskip

\noindent\textbf{Step~3.} In the above step, we have shown that
$$  w_\lambda(x,t) \geq  0, \;\forall (x,t)\in\Sigma_\lambda\times\R,\;\forall \lambda\leq0.$$
Now, we will further prove that the strict inequality holds:
\be \label{mpstrict}  w_\lambda(x,t) >  0, \;\forall (x,t)\in\Sigma_\lambda\times\R,\;\forall \lambda<0. \ee
Otherwise, there exists some $\bar{\lambda}<0$ and a point $(x^o, t_o)$ in $\Sigma_{\bar{\lambda}} \times \R$, such that
$$ w_{\bar{\lambda}} (x^o, t_o) = \min_{\Sigma_{\bar{\lambda}} \times \R} w_{\bar{\lambda}} (x,t) = 0.$$
At this minimum point, we evaluate $(\partial_t-\Delta)^s w_{\bar{\lambda}}$ in two ways.
First, from the integral representation \eqref{nonlocaloper} and the antisymmetry of $w_{\bar{\lambda}}$, we have
\begin{align}\label{step3_int}
 (\partial_t-\Delta)^s w_{\bar{\lambda}} (x^o, t_o)
 &=C_{n,s}\int_{-\infty}^{t_o}\int_{\R^{n}}\frac{-w_{\bar{\lambda}}(y,\tau)}{(t_o-\tau)^{\frac{n}{2}+1+s}}e^{-\frac{|x^o-y|^2}{4(t_o-\tau)}}\operatorname{d}\!y\operatorname{d}\!\tau \nonumber\\
 &=C_{n,s}\int_{-\infty}^{t_o}\int_{\Sigma_{\bar{\lambda}}}\frac{-w_{\bar{\lambda}}(y,\tau)}{(t_o-\tau)^{\frac{n}{2}+1+s}}e^{-\frac{|x^o-y|^2}{4(t_o-\tau)}}\operatorname{d}\!y\operatorname{d}\!\tau \nonumber\\
 &\qquad + C_{n,s}\int_{-\infty}^{t_o}\int_{\Sigma_{\bar{\lambda}}}\frac{-w_{\bar{\lambda}}(y^{\bar{\lambda}},\tau)}{(t_o-\tau)^{\frac{n}{2}+1+s}}e^{-\frac{|x^o-y^{\bar{\lambda}}|^2}{4(t_o-\tau)}}\operatorname{d}\!y\operatorname{d}\!\tau \nonumber\\
 &=C_{n,s}\int_{-\infty}^{t_o}\int_{\Sigma_{\bar{\lambda}}}\frac{-w_{\bar{\lambda}}(y,\tau)}{(t_o-\tau)^{\frac{n}{2}+1+s}}\Big[e^{-\frac{|x^o-y|^2}{4(t_o-\tau)}}-e^{-\frac{|x^o-y^{\bar{\lambda}}|^2}{4(t_o-\tau)}}\Big]\operatorname{d}\!y\operatorname{d}\!\tau \nonumber\\
 &\le 0,
\end{align}
because $w_{\bar{\lambda}}\ge 0$ in $\Sigma_{\bar{\lambda}}\times\R$ and $|x^o-y|<|x^o-y^{\bar{\lambda}}|$ implies the bracket is nonnegative.

Second, from the equation \eqref{2.0} and the strict decrease of $R(|x|)$, we obtain
\begin{align}\label{step3_eq}
 (\partial_t-\Delta)^s w_{\bar{\lambda}}(x^o,t_o)
 &= R(|(x^o)^{\bar{\lambda}}|)f(u_{\bar{\lambda}}(x^o,t_o)) - R(|x^o|)f(u(x^o,t_o)) \nonumber\\
 &= \big[R(|(x^o)^{\bar{\lambda}}|)-R(|x^o|)\big]f(u_{\bar{\lambda}}(x^o,t_o)),
\end{align}
where we used $u_{\bar{\lambda}}(x^o,t_o)=u(x^o,t_o)$ because $w_{\bar{\lambda}}(x^o,t_o)=0$.
Since $\bar{\lambda}<0$ and $x^o\in\Sigma_{\bar{\lambda}}$, we have $|(x^o)^{\bar{\lambda}}|<|x^o|$, and the strict monotonicity of $R$ yields $R(|(x^o)^{\bar{\lambda}}|) > R(|x^o|)$.  Moreover $f>0$, hence the right-hand side of \eqref{step3_eq} is strictly positive.

This contradicts the non-positivity obtained in \eqref{step3_int}.  Hence such a point $(x^o,t_o)$ cannot exist, and \eqref{mpstrict} must hold.  Consequently, for each fixed $t\in\R$, $u(x,t)$ is strictly increasing in $x_1$ for $x_1<0$.

Since the choice of the \(x_1\)-direction is arbitrary, the same argument can be performed in any direction \(e\in S^{n-1}\). Hence \(u(\cdot,t)\) is symmetric with respect to every hyperplane passing through the origin. Therefore \(u(\cdot,t)\) is radially symmetric about the origin. The strict monotonicity obtained before the limiting position further implies that the radial profile is strictly decreasing in \(r=|x|\).

This completes the proof of Theorem \ref{thm1}.

\end{proof}

\section*{Acknowledgments}
The work of the author is partially supported by the National Natural Science Foundation of China (NSFC Grant No.\ 12301264) and the Henan Natural Science Foundation Project (232300421347). The author would like to thank Professor Wenxiong Chen (Yeshiva University) for helpful discussions and valuable suggestions during the preparation of this paper.

\section*{Data availability statement}
Data sharing is not applicable to this article as no new data were created or analyzed in this study.

\bibliographystyle{plain}
\bibliography{references}

\end{document}